\documentclass[12pt]{article}
\usepackage{graphicx}
\usepackage{epsfig}
\usepackage{amsthm}
\usepackage{amsmath}
\usepackage{latexsym}
\usepackage{amsfonts}
\usepackage{amssymb}

\newtheoremstyle{theorem}
{10pt} 
{10pt} 
{\sl} 
{\parindent} 
{\bf} 
{. } 
{ } 
{} 
\theoremstyle{theorem}
\newtheorem{theorem}{Theorem}
\newtheorem{corollary}[theorem]{Corollary}

\newtheoremstyle{defi}
{10pt} 
{10pt} 
{\rm} 
{\parindent} 
{\bf} 
{. } 
{ } 
{} 
\theoremstyle{defi}
\newtheorem{definition}[theorem]{Definition}

\newtheoremstyle{defi}
{10pt} 
{10pt} 
{\rm} 
{\parindent} 
{\bf} 
{. } 
{ } 
{} 
\theoremstyle{defi}
\newtheorem{remark}[theorem]{Remark}

\newtheoremstyle{defi}
{10pt} 
{10pt} 
{\rm} 
{\parindent} 
{\bf} 
{. } 
{ } 
{} 
\theoremstyle{defi}
\newtheorem{lemma}[theorem]{Lemma}

\newtheoremstyle{defi}
{10pt} 
{10pt} 
{\rm} 
{\parindent} 
{\bf} 
{. } 
{ } 
{} 
\theoremstyle{defi}
\newtheorem{proposition}[theorem]{Proposition}

\begin{document} 

\title{Reynolds' Limit Formula for Dorodnitzyn's 
Atmospheric Boundary Layer Model in Convective Conditions}

\author{C. V. Valencia-Negrete $^1$ , 
C. Gay-Garc\'{i}a $^2$,
A. A. Carsteanu $^3$\\[6pt]
$^1$ $^3$ Superior School of Physics and Mathematics (ESFM-IPN) \\
National Polytechnic Institute \\
Mexico City - 07738, MEXICO\\
e-mail: ohbicarla@gmail.com\\[6pt]
$^2$ Centre of Atmospheric Sciences (CCA-UNAM)\\
National Autonomous University of Mexico\\
Mexico City - 04510 M\'{E}XICO \\ e-mail: cgay@unam.mx}

\maketitle

\begin{abstract}
Atmospheric convection is an essential aspect of 
atmospheric movement, and it is a source of errors 
in Climate Models. Being able to generate appro\-ximate 
limit formulas and compare the estimations they produce, 
could give a way to reduce them. In this article, 
it is shown that it is enough to assume that the velocity's
$L^2$-norm is bounded, has locally integrable, $L^1_{loc}$,
weak partial derivatives up to order two, and a negligible 
variation of its first velocity's coordinate in direction 
parallel to the surface, to obtain a Reynolds' limit formula 
for a Dorodnitzyn's compressible gaseous Boundary Layer in 
atmospheric conditions.

\medskip



\end{abstract}

\section{Introduction}

A suitable approximate model for the air near the Earth's 
surface could tie both the free-stream velocity and the 
no slip condition. In the Theorem~\ref{theo:2.4}, it will 
be shown that there is a Reynolds' limit formula:
\begin{eqnarray*}
   f
    \hspace{2pt}
    \frac{\partial^2 u}{\partial y^2}
&=&
\frac{\partial f}{\partial y}
\hspace{2pt}
\frac{\partial u}{\partial y},
\end{eqnarray*}
for a Dorodnitzyn's compressi\-ble Boundary Layer,
where $u$ is the first veloci\-ty's component,
$f=\left[1-\left(u^2/2i_0\right)\right]^{-6/25}$,
$y$ denotes the height, and $i_0$ is constant.  
In order to do so, we find an estimate, independent 
of the domain's scale:
\begin{eqnarray*}
\Vert \nabla F^{\epsilon} 
 \Vert_{L^2\left(\boldsymbol{\Omega};\mathbb{R}^2\right)}
 &\leq&
\frac{c_2\hspace{2pt}U^3}{2\hspace{2pt}C}.
\end{eqnarray*}
for the $L^2$-norm of a corresponding incompressible 
vector field $F^{\epsilon}$, where $U$ is the air's velocity 
over the Boundary Layer, and $C$ is a constant set out by 
the rest of the boundary conditions given to the initial problem.

The solution procedure consists of three main steps. First, 
the application of Bayada and Chambat's change of variables
transforms the original problem to an adimensional model 
where the e\-ffect of the small parameter of proportion,
\[\epsilon =\max \left\{h(x)\hspace{4pt}|\hspace{4pt}x\in [0,L]\right\}/L,\]
\noindent
on each term, is explicit. Then, an adaptation of Dodordnitzyn's
technique is applied to present it in an incompressible form, 
where Majda's Energy Method is used to obtain a bound that is
independent of $\epsilon$ for the $L^2$-norm of the incompressible
gradient. Finally, we show that the family of solutions to the
adimensional problem, indexed by the small parameter 
$\epsilon$ is contained in a bounded set of a Sobolev space. 
Consequently, the Rellich-Kondrachov Compactness Theorem implies
that the sequence of solutions has a subsequence that converges
uniformly in the space $L^2\left(\boldsymbol{\Omega}\right)$
when the parameter $\epsilon$ tends to zero.


\subsection{Motivation}\label{subsec:1.1}

There is a need to lower biases in continental warmth to obtain 
better atmosphere models G. M. Martin et al. \cite[p. 725]{HadGEM2}. The release of 
energy to the atmosphere by convective parcels contributes to these 
errors. Its calculation has histori\-cally been a way to reduce 
inaccuracies in surface temperature descriptions K. St\"{u}we \cite[p. 59]{Stuwe}. 
A temperature diffe\-rence between a specific surface in contact 
with a gas and its surrounding neighborhood is the origin of a 
vertical draft of air, a natural convection air parcel. 
A sudden expansion of the gas in touch with the increased 
temperature gives a drop in its density, which in turn makes it 
lighter B. R. Morton et al. \cite{Morton56} and A. Bouzinaoui et al.
\cite{bouz07}. However, ascending air acceleration is modeled by
compressi\-ble Navier-Stokes equations P.-L. Lions
\cite{LionsVol1} and 
F. Boyer et al. \cite{BoyerFabrie2013}.
The sugges\-tion of this work is to overcome this difficulty by 
looking for Reynolds' limit formulas, deduced from compressible 
Boundary Layer models. In this article, a first Reynolds' limit 
formula is found for the Dorodnitzyn's ideal gas and constant total
energy Boundary Layer model, which admits an incompressible
adimensional presentation where the evolution parameter problem 
is stated and the convective non linear term estimated through its 
\emph{free-stream} velocity value.


\subsection{Statement of the Problem}\label{subsec:1.2}

An atmospheric gas is a \emph{newtonian fluid}, which implies 
the use of compressible Navier-Stokes equations 
P.-L. Lions \cite{LionsVol1}.
If, instead of considering a Boundary Layer, a two-dimensional
incompressible Navier-Stokes model is applied to study the 
behaviour of a liquid in contact with a solid surface, then 
there exist a smooth solution for each given viscosity value.
For a fixed initial condition, a set of viscosity values has a
corresponding family of well defined classical solutions.
When the viscosity tends to zero, this family of solutions
converges to an Euler's Equations solution with the same 
initial condition 
A. J. Majda et al. \cite{MajBert2002}. 

However, even in the simplest case of an incompressible flow
whose vorticity is zero everywhere on its domain, an Euler's 
solution satisfying the condition of null velocity at $\Gamma_0$, 
has a null velocity throughout 
the whole domain C. V. Valencia \cite[p. 19]{Val2014}.
Therefore, there are no two-dimensional Euler solutions with 
zero vorticity that comply with both the positive hori\-zontal
component of velocity at the top of the domain and the no slip 
condition at its bottom 
H. Schlichting et al. (\cite{SchGers2009} p. $145$).
This motivates the statement of a Boundary Layer model to more
appropriately depict this phenomenon. Moreover, nume\-rical
approximations of bounda\-ry layer solutions describe velocity 
profiles similar to those found in reality 
H. Schlichting \cite[p. 143]{Sch1955}.

In $1935$, Adolf Busemann \cite{Bus1935} proposed the first
compressible Boundary Layer model to represent the behaviour of 
a gas with u\-pper outflow velo\-city smaller than the velocity of
sound, and Prandtl number equal to one.
In his model, pressure terms are discarded, but temperature,
viscosity, and density vary in accordance with ideal gas
empirical properties to more accurately describe an atmospheric
boundary layer moving over a surface. He presents tempe\-rature 
as a function of velocity, and employs it to describe the rest 
of the state variables in terms of velocity as well.

Busemann's model considered a power-law between viscosity and
temperature whose exponent was later corrected in Theodore 
von K\'{a}rm\'{a}n and Hsue-Shen Tsien \cite{VKar1938} $1938$'s 
article, where they developed a different method 
of solution for the same problem. Less than a decade
later, in $1942$, Anatoly Alekseevich Dorodnitsyn \cite{Dorod42} 
postulated a similar model, but allo\-wed pressure to vary with $x$,
which could imply the Boundary Layer to be separated from the
surface. In this work, he defined several changes of varia\-bles.
The first one of these allowed him to write the compressible model
as an incompressible system. Here, we adapt this coordinates' 
change to a similar but not rectangular adimensional domain that
will be obtained from $\boldsymbol{\Omega_h}$, and defined in
Theorem~\ref{theo:2.1}.

Limit formulas for a small parameter of proportion have their origin
in Osborne Reynolds' \cite{Reynolds01011886} article 
``\emph{On the Theory of Lubrication and Its Application to Mr.
Beauchamp Tower's Experiments, Including an Experimental
Determination of the Viscosity of Olive Oil}", published in $1886$. 
Reynolds' Formula was extensively used without a formal proof 
that it was indeed Navier-Stokes Equations' limit when the small
parameter of proportion between the domain's height and its length
tends to zero. This was accomplished a hundred years later by 
Guy Bayada and Mich\`{e}le Chambat \cite{Bayada1986}
for Stokes' Equations. 

In $2009$, Laurent Chupin and R\`{e}my Sart \cite{ChupinSart}
successfully showed, through an application of Didier Bresch 
and Beno\^{i}t Desjardin's Entropy Methods, that the compressible
Reynolds equation is an approximation of compressible Navier-Stokes
equations. For a thin domain filled with gas, the authors
mention that there appears to be only one result 
of this type of problem. This is due to Eduard Marusic-Paloka and
Maja Starcevic \cite{MARUSICPALOKA20104565}
\cite{MARUSICPALOKA2005534}. Marusic-Paloka and Starcevic show 
the convergence of a two-dimensional compressible Stokes Equations. 

In the literature, it doesn't seem to exist a small parameter
asymptotic analysis for a compressible gaseous Boundary Layer model
with a convective non linear term, such as Dorodnitzyn's Model, 
nor an adaptation of Dorodnitzyn's change of variables 
to this particular domain's shape to find a limit formula for a
compressible case in terms of an incompressible expression. 
The main result of this study is stated in Theorem~\ref{theo:2.3} 
and proved in Subection~\ref{subsec:2.3}. Meanwhile, it can be
expressed by the following assertion: Dorodnitzyn's Model may 
be approximated by a limit formula.

\subsection{The Domain}\label{subsec:1.3}

Laminarity ---and therefore two-dimensionality of the domain---
in the liquid's movement when it is in contact with a solid surface 
is a supposition based on experimental observations 
T. von K\'{a}rm\'{a}n et al. \cite{VKar1938} and S. Goldstein
\cite{Goldstein48}, and it is still regarded as a good assumption 
to describe it at an initial stage of a Boundary Layers' motion
K. Gersten \cite[p. 11]{Gersten2009} and S. Goldstein
\cite{Goldstein48}. Here, the Boundary Layer is represented as a
two-dimensional slice where the convective bubble is beginning to
form although it has not yet separated from the surface, and it is
slightly different from the rectangle that constitutes the domain 
in Dorodnitzyn's model.

\begin{figure}
\includegraphics[scale=0.5]{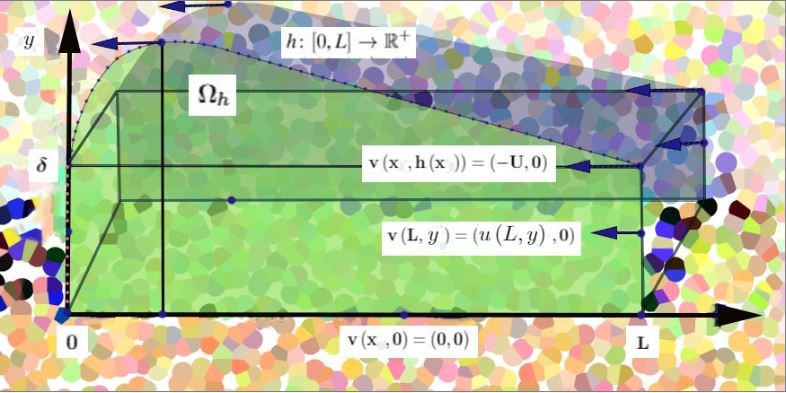}
\centering
\caption{The Domain $\boldsymbol{\Omega_h}$}
\label{fig:Domain}
\end{figure}
\vspace{3mm}

\begin{definition}\label{defi1.1}
Let $h\colon [0,L] \to (0,\infty)$ be a smooth function 
such that $h(0)=h(L)=\delta$. The curve $h$ is assumed to be 
twice differentiable in the interval $(0,L)$ with well defined
continuous extensions for itself and its derivatives to $\{0\}$ 
and $\{L\}$, \emph{i.e.} 
$h\in C^2\left([0,L];\left(0,\infty\right)\right)$, 
and to have only one critical point which is a maximum. 
Moreover, suppose $L>0$. Then, the domain is denoted as:
\begin{eqnarray*}
    \boldsymbol{\Omega_h} \colon = \{(x,y)\in \mathbb{R}^2 \hspace{4pt} | \hspace{4pt} 0<x<L \hspace{4pt} \& \hspace{4pt} 0<y<h(x)\}.
\end{eqnarray*}
\end{definition}

The domain's topological boundary, $\partial \boldsymbol{\Omega_h}$,
is drawn by the union of the segments:
$\Gamma_0 = \left\{\left(x,0\right) \in \mathbb{R}^2 \hspace{4pt};\hspace{4pt} 0\leq x\leq L\right\}$,
$\Lambda_0 = \{\left(0,y\right)\in \mathbb{R}^2;$ 
$0\leq y \leq \delta \}$,
$\Lambda_L = \left\{\left(L,y\right) \in \mathbb{R}^2\hspace{4pt};\hspace{4pt}0\leq y \leq \delta \right\}$,
and the curve
\[\Gamma_h = \left\{\left(x,h(x) \right)\in \mathbb{R}^2
\hspace{4pt}; \hspace{4pt}
0\leq x\leq L \right\}.\]

\begin{remark}
The vector $\mathbf{-e_1}=(-1,0) \in \mathbb{R}^2$ depicts 
the wind's direction above the Boundary Layer $\boldsymbol{\Omega_h}$. 
Likewise, $\mathbf{e_3}=(0,1) \in \mathbb{R}^2$ portrays the
direction from the Earth's surface to its atmosphere.
Similarly, the length $L>0$ is a fixed real number which 
represents the distance covered by the \emph{free-stream} in
direction $\mathbf{-e_1}=(-1,0)$ over $\Gamma_h$.
On the other hand, continuation of trajectories in the 
Boundary Layer is broken if for some $x\in [0,L]$ 
there is a pre\-ssure drop that generates a lift, a separation 
of the volume of the air from the surface. At that moment, the
phenomenon's description in terms of a fixed domain 
is no longer possible.
\end{remark}

\subsection{Dorodnitzyn's Model Equations}\label{subsec:1.4}

\begin{definition}\label{defi1.2}
Let $\boldsymbol{\Omega_h}$ be as in Definition~\ref{defi1.1}, 
$\rho \in L^1\left( \boldsymbol{\Omega_h};
\left(0,\infty\right)\right)$
be the \emph{density}; the \emph{velocity},
$\mathbf{v}=(u,v)\in L^2\left(\boldsymbol{\Omega_h};\mathbb{R}^2\right)
\cap L^1_{loc}\left(\boldsymbol{\Omega_h};\mathbb{R}^2\right)$;
the \emph{absolute temperature},
$T\in L^1_{loc}\left(\boldsymbol{\Omega_h}; \left(0,\infty \right)\right)$;
the \emph{pressure},
$p\in L^1_{loc}\left(\boldsymbol{\Omega_h}\right)$;
the \emph{dynamic viscosity}, 
$\mu \in L^1_{loc}\left(\boldsymbol{\Omega_h}\right)$;
and the \emph{thermal conductivity} be 
$\kappa \in L^1_{loc}\left(\boldsymbol{\Omega_h}\right)$;
all with well defined first order weak partial derivatives,
locally integrable in the Lebesgue sense, \emph{i.e.} in
$L^1_{loc}\left(\boldsymbol{\Omega_h}\right)$.
\end{definition}

Dorodnitzyn's model is formed by seven equations given for 
the seven variables
$\rho$, $u$, $v$, $T$, $p$, $\kappa$, $\mu$,
described above. The first three come from the conservation laws 
of Newtonian fluids: the stationary Conservation of Mass Law 
F. Boyer et al. \cite{BoyerFabrie2013}, Eq. (\ref{eq:1.1}), 
the compressible Boundary Layer Conservation of Momentum Law 
A. Dorodnitzyn \cite{Dorod42}, Eq. (\ref{eq:1.2}), and the
simplified Conservation of Ener\-gy per Unit Mass Law, 
Eq. (\ref{eq:1.9}), that is obtained in Proposition~\ref{prop1.1}
from an application of Luigi Crocco's \cite{Crocco32} procedure 
to the, statio\-nary and approximated, Conservation of Energy Law
stated below as Eq. (\ref{eq:1.3}).

Consider:
\noindent
  \begin{align}
  \frac{\partial \hspace{2pt}\left(\rho \hspace{2pt} u\right)}{\partial x}
  +
  \frac{\partial \hspace{2pt}\left(\rho \hspace{2pt} v\right)}{\partial y}
  & \hspace{2pt} = \hspace{2pt}
  0 \hspace{2pt}; \label{eq:1.1}\\
  \rho \left( u \hspace{2pt} \frac{\partial u}{\partial x}+
  v \hspace{2pt} \frac{\partial u}{\partial y}
  \right)
  & \hspace{2pt} =  \hspace{2pt}
  - \hspace{2pt} \frac{\partial p}{\partial x}
  +
  \frac{\partial}{\partial y}
  \left( \mu \hspace{2pt}
  \frac{\partial u}{\partial y}
  \right); 
  \hspace{7pt}\mbox{and} \label{eq:1.2}\\
  \rho \hspace{2pt}  
\left[ u \hspace{2pt}
\frac{\partial \hspace{2pt}\left(c_p \hspace{2pt} T\right)}{\partial x} 
+v \hspace{2pt}
\frac{\partial \hspace{2pt}\left(c_p\hspace{2pt} T\right)}{\partial y}\right]
& \hspace{2pt} = \hspace{2pt}
\frac{\partial}{\partial y} \left[\kappa \hspace{2pt} \frac{\partial T}{\partial y}\right]
+
\mu \hspace{2pt}
\left(
\frac{\partial u}{\partial y}
\right)^2
+
\frac{\partial p}{\partial t},
\label{eq:1.3}
\end{align}
where $c_p$ is the \emph{specific heat transfer coefficient 
at constant pressure}.

The next four are Ideal Gases properties and empirical laws.
In general, the dynamic viscosity $\mu$ satisfies the
proportionality relation
\begin{eqnarray*}
    Pr&=&\frac{c_p \hspace{2pt} \mu}{\kappa}  
\end{eqnarray*}
\noindent
for a \emph{thermal conductivity} $\kappa$
and a Prandtl number $Pr$. In this case,
assume $Pr=1$. This is:
\begin{equation}\label{eq:1.4}
   1= \frac{c_p \hspace{2pt} \mu}{\kappa};
\end{equation}
\noindent
and, the \emph{Equation of State} K. Saha \cite[p. 24]{Saha}, 
\begin{equation}\label{eq:1.5}
    p\hspace{2pt}V 
    \hspace{2pt}=\hspace{2pt}
    n\hspace{2pt }R^* \hspace{2pt} T; 
\end{equation}
\noindent
where $R^*$ is the \emph{Universal Gas Constant}, 
$n$ is the number of moles in a volume $V$, and
$V=V\left(B_r\right)=\int \!\!\! \int \!\!\! \int_{B_r}d\mathbf{x}$ where
\[B_r=\{\mathbf{x}=(x,y,z)\in \mathbb{R}^3 \hspace{4pt};\hspace{4pt} \Vert\mathbf{x}-\mathbf{\hat{x}}\Vert <r\}\subset \mathbb{R}^3,\]
\noindent
for a given point $\mathbf{\hat{x}} \in \boldsymbol{\Omega_h} \cap  \mathbb{R}^3$
and a value $r>0$ such that $\boldsymbol{\Omega_h} \subset B_r$.
This last equation is also used by Dorodnitzyn in the form:
\begin{eqnarray}\label{eq:1.6}
\rho &=& \frac{p}{RT},  
\end{eqnarray}
for $R=R^*/M$, where 
$M$ is the \emph{molecular weight} of the gas. 

The \emph{adiabatic polytropic atmosphere} 
O. G. Tietjens \cite[p. 35]{Tiet} is a relation:
\begin{equation}\label{eq:1.7}
       p\hspace{2pt} V^{b}  
       \hspace{2pt}= \hspace{2pt}
       c; 
\end{equation}
where $b\cong 1.405$, $c$ are fixed constants, and $V$ has 
the value des\-cribed above.

Finally, given two values $\mu_0$ and $T_0$ of $\mu$ and $T$
at the same point $(x_0,y_0) \in \boldsymbol{\Omega_h}$,
there is a \emph{Power-Law}  A. J. Smits et al.
\cite[p. 46]{SmitsDussauge2006}:
\begin{equation} \label{eq:1.8}
 \frac{\mu}{\mu_0}  
 \hspace{2pt} = \hspace{2pt} 
  \left( \frac{T}{T_0}\right)^{\frac{19}{25}}.
\end{equation}

\begin{remark}
First of all, when the air flow moves over a plane surface,
has a velocity lower than the velocity of sound, and the surface 
has a homogeneous temperature, the Prandtl number is equal to $1$ 
H. Schlichting et al. \cite[p. 215]{SchGers2009}, 
Eq. (\ref{eq:1.4}),
and $c_p\hspace{2pt}\mu$ replaces $\kappa$ in Eq. (\ref{eq:1.3}).
Second, a gas in the range of temperatures and densities found in 
the Earth's atmosphere fulfills the premises discovered for 
an Ideal Gas P.-L. Lions \cite[p. 8]{LionsVol1}, such as the
Equation of State,
Eq. (\ref{eq:1.5}). Moreover, when air moves in a convective parcel,
the process is fast enough to expect that there will not be a 
heat transfer between the gas within the convective draft
and its environment. Thus, adiabatic conditions imply another
association, known as an adiabatic polytropic 
atmosphere O. G. Tietjens \cite[p. 35]{Tiet}.
Addi\-tionally, in a  tempe\-rature range of $[150,500]$ Kelvin,
there is a Power-Law between dynamic viscosity 
and $T$ A. J. Smits et al. \cite[p. 46]{SmitsDussauge2006}.
\end{remark}
  
One can follow L. Crocco's \cite{Crocco32} procedure to find
a Conservation of Energy Law from which Dorodnitzyn's model 
equation Eq. (\ref{eq:1.9}) is deduced, and find that it is
equivalent to Eq. (\ref{eq:1.3}) when the Prandtl number is 
equal to $1$, Eq. (\ref{eq:1.4}), as it is outlined in the 
following paragraph.

\begin{proposition}\label{prop1.1}
Let $\rho, u,v,T,p,\kappa,\mu$ be as they were described in
Definition~\ref{defi1.2}. Then, they satisfy Eq. (\ref{eq:1.3})
if and only:
\begin{eqnarray}\label{eq:1.9}
     \rho \hspace{2pt}  
\left[ u \hspace{2pt}
\frac{\partial }{\partial x} 
+v \hspace{2pt}
\frac{\partial}{\partial y}\right]
\left(
c_p \hspace{2pt} T+\frac{u^2}{2}
\right)
\hspace{2pt}& = &\hspace{2pt}
\frac{\partial}{\partial y} \left[ \mu \hspace{2pt} 
\frac{\partial \hspace{2pt} }{\partial y}
\left(
c_p \hspace{2pt} T
+
\frac{u^2}{2}
\right)
\right].
\end{eqnarray}
\end{proposition}

\begin{proof}
First, Eq. (\ref{eq:1.4}) allows to make the substitution
$\kappa=c_p \hspace{2pt}\mu$
in the right side or Eq. (\ref{eq:1.3}). This way one can arrive at:
\begin{equation}\label{eq:1.10}   \rho \hspace{2pt}  
\left[ u \hspace{2pt}
\frac{\partial \hspace{2pt}\left(c_p \hspace{2pt} T\right)}{\partial x} 
+v \hspace{2pt}
\frac{\partial \hspace{2pt}\left(c_p\hspace{2pt} T\right)}{\partial y}\right]
\hspace{2pt} = \hspace{2pt}
\frac{\partial}{\partial y} \left[ \mu \hspace{2pt} 
\frac{\partial \hspace{2pt} \left(c_p \hspace{2pt} T\right)}{\partial y}\right]
+
\mu \hspace{2pt}
\left(
\frac{\partial u}{\partial y}
\right)^2
+
\frac{\partial p}{\partial t}.  
\end{equation}
Also, the product of the first velocity coordinate $u$
and Eq. (\ref{eq:1.2}) gives:
\begin{equation}\label{eq:1.11}
    \rho \left[ 
    u \hspace{2pt} \frac{\partial }{\partial x}
   \left( \frac{u^2}{2} \right)
   +
  v \hspace{2pt} \frac{\partial}{\partial y}
  \left( \frac{u^2}{2} \right)
  \right]
 \hspace{2pt} =  \hspace{2pt}
u \hspace{2pt}
  \frac{\partial}{\partial y}
  \left( \mu \hspace{2pt}
  \frac{\partial u}{\partial y}
  \right)
    - u \hspace{2pt} \frac{\partial p}{\partial x}.
\end{equation}
Finally, Eq. (\ref{eq:1.9}) is obtained from the additon of 
Eq. (\ref{eq:1.10}) and (\ref{eq:1.11}) because 
$\partial p/\partial t 
\hspace{2pt}= \hspace{2pt}
(\partial p/\partial x)
\hspace{2pt}
(\partial x/\partial t)
\hspace{2pt}= \hspace{2pt}
u \hspace{2pt}
(\partial p/\partial x)
$.
\end{proof}

\begin{remark}
It is possible to notice in Eq. (\ref{eq:1.9}) that in Dorodnitzyn's
model, the \emph{kinetic energy} generated by the velocity
coordinate $v$ in the orthogonal direction to the surface is taken
as negligible; the \emph{total ener\-gy per unit mass}, $E=c_pT+u^2/2$,
is considered the addition of the 
\emph{kinetic energy per unit mass} $u^2/2$
and the \emph{internal energy} in terms of \emph{specific enthalpy} $e=c_pT$.
\end{remark}

\subsection{Dorodnitzyn's Model Boundary Conditions}
\label{subsec:1.5}

The velocity at the upper top $\Gamma_h$ of 
$\boldsymbol{ \partial \Omega_h}$
is called the \emph{free-stream velocity}. Let:
\begin{eqnarray}\label{eq:1.12}
\mathbf{v}|_{\Gamma_h}&=&(-U,0),
\end{eqnarray}
\noindent
for a strictly positive constant real value $U>0$. Also, 
the velocity value at the lower lid $\Gamma_0$ is:
\begin{eqnarray}\label{eq:1.13}
\mathbf{v}|_{\Gamma_0}&=&(0,0).
\end{eqnarray}

Similarly, a constant \emph{free-stream temperature}, 
\begin{eqnarray}\label{eq:1.14}
T|_{\Gamma_h}&=&T_h>0,
\end{eqnarray}
\noindent
and a homogeneous \emph{free-stream dynamic viscosity} value
\begin{eqnarray}\label{eq:1.15}
\mu|_{\Gamma_h}&=&\mu_h>0,
\end{eqnarray}
\noindent
are given in $\Gamma_h$.

Furthermore, there are periodic velocity conditions at the 
vertical segments of the boundary, $\Lambda_0$ and $\Lambda_L$,
described in the Definition~\ref{defi1.1}. This is: 
For all $y \in (0,\delta)$,
 \begin{eqnarray}\label{eq:1.16}
 \left(u\left(0,y\right),0\right)
 & = & \left(u\left(L,y\right),0\right).
\end{eqnarray}
\noindent
Finally, we have a Neumann condition for $T$: For all $x \in [0,L]$,
 \begin{eqnarray}\label{eq:1.17}
\frac{\partial T}{\partial y}(x,0)
 & = & 0.
\end{eqnarray}

\begin{remark}
This last condition represents an adiabatic wall in the surface
$\Gamma_0$. If the wind's velocity is less than the velocity of
sound, the gas adheres to the solid surface 
T. von K\'{a}rm\'{a}n et al. \cite{VKar1938} and
A. J. Smits et al. \cite[p. 52]{SmitsDussauge2006}. 
This is called the \emph{no slip} condition,
as seen in Eq. (\ref{eq:1.13}). On the other hand, there is a
logarithmic wind velocity profile
on the Earth's troposphere that depends on the type of atmosphere,
and is not valid close to the Earth's surface, but provides a
boundary condition $U$ at the upper top $\Gamma_h$ of
$\boldsymbol{\Omega_h}$. For example, the classical Fleagle 
and Businger's \cite[p. 274]{Fleagle} Atmospheric Physics book 
reports a horizontal velocity measurement of $4$ $m/s$ at a height
of $0.4$ $m$, $\mathbf{v}(0,0.4)=(4,0)$, in an \emph{unstable}
atmosphere at O'Neill, Nebraska on $19$ August $1953$. Moreover,
this value and the free-stream temperature determine that of the
surface temperature, as will be shown in the following
Lemma~\ref{lem2.1}, Eq. (\ref{eq:2.20}). Similarly, the pressure
$p|_{\Gamma_0}=p_0$ can be known from $U$ and $T_h$ through 
Eq. (\ref{eq:1.5}) and (\ref{eq:2.20}). Once the density
is expressed in terms of the velocity $u$, as in Lemma~\ref{lem2.2},
$\rho|_{\Gamma_0}=\rho_0$ can be calculated. Finally, 
Eq. (\ref{eq:1.4}) provides a way to obtain $\mu_h$ from a 
surface value of $\kappa_h$ given by the material.
\end{remark}

\section{Limit Formula}
\label{sec:2}

\subsection{Adimensional Model}
\label{subsec:2.1}

\begin{lemma}\label{lem2.1}
Let $\rho, u,v,T,p,\kappa,\mu$ be as in Definition~\ref{defi1.2}.
If the no slip condition (\ref{eq:1.13}) is satisfied, then 
Eq. (\ref{eq:1.9}) has a constant solution $E=c_p\hspace{2pt}T_h+U^2/2$
in the domain $\boldsymbol{\Omega_h}$, described in Definition~\ref{defi1.1}, that fulfills the remaining 
boundary conditions (\ref{eq:1.12}), (\ref{eq:1.14}), 
and (\ref{eq:1.17}) given for $u$, $v$, and $T$.
\end{lemma}

\begin{proof}
It is enough to substitute the constant value $E=c_p\hspace{2pt}T_h+U^2/2$
in Eq. (\ref{eq:1.9}) to see that both sides become zero.
Because $E=c_p\hspace{2pt}T+u^2/2$, this allows us to express 
the absolute temperature in the form 
\begin{equation}\label{eq:2.19}
   T(u)=T_h+\frac{1}{2c_p}\left(U^2-u^2\right).
\end{equation}
The boundary conditions (\ref{eq:1.12}) and (\ref{eq:1.14})
are verified by construction. If $y=0$, Eq. (\ref{eq:2.19}) 
and the no slip condition (\ref{eq:1.13}) imply that:
\begin{equation}\label{eq:2.20}
    T|_{\Gamma_0}=T_h+\left(1-\frac{U^2}{2c_p}\right).
\end{equation}
Thus, $T|_{\Gamma_0}=T_0>0$, and the boundary condition (\ref{eq:1.17}) is fulfilled.
\end{proof}

\begin{corollary}\label{cor2.1}
Under the same assumptions, where the free-stream temperature
$T_h>0$, as is stated in (\ref{eq:1.14}), $T$ can be seen in
terms of $T_0$ as:
\begin{eqnarray}\label{eq:2.21}
    T(u)&=&T_0\hspace{2pt}
    \left(1-\frac{u^2}{2c_p\hspace{2pt}T_0}\right).
\end{eqnarray}
\end{corollary}

\begin{proof}
The previous Lemma~\ref{lem2.1} shows that the total energy $E$ 
has a constant value throughout the domain. We can use the no slip
condition (\ref{eq:1.13}) in the expression
$E=c_p\hspace{2pt}T+u^2/2$ to obtain a new way to calculate it as 
$E=c_p\hspace{2pt}T_0$. Hence, $c_p\hspace{2pt}T_0=c_p\hspace{2pt}T+u^2/2$, 
and we get Eq. (\ref{eq:2.21}).
\end{proof}

\begin{remark}\label{rem2.1}
By Definition~\ref{defi1.2}, the absolute temperature $T>0$ 
in the domain $\boldsymbol{\Omega_h}$ as described in 
Definition~\ref{defi1.1}. Additionally $c_p>0$. In consequence,
the total energy per mass unit $i_0\colon =c_pT_0=c_pT+u^2/2$ is
strictly bigger than the kinetic energy $u^2/2$ gene\-rated by 
the first velocity's component. Therefore the difference
   $1-\left( 
      u^2/2i_0\right) \not=0$
in $\boldsymbol{\Omega_h}$.
\end{remark}

\begin{lemma}\label{lem2.2}
Once again, let $\rho, u,v,T,p,\kappa,\mu$ be as in
Definition~\ref{defi1.2}. Suppose that Eq. (\ref{eq:1.3}),
(\ref{eq:1.4}), (\ref{eq:1.5}), (\ref{eq:1.7}) and (\ref{eq:1.8})
are satisfied by $\rho$, $u$, $v$, $T$, $p$, $\kappa$, and $\mu$ 
in $\boldsymbol{\Omega_h}$ with the boundary conditions
(\ref{eq:1.12}), (\ref{eq:1.13}), (\ref{eq:1.14}), 
and (\ref{eq:1.15}). Then:
\noindent
\begin{align}
  p(u)
& \hspace{2pt} = \hspace{2pt}
     c_1
      \hspace{2pt}
    \left[1-\left(u^2/2i_o\right)\right]^{\frac{b}{(b-1)}}; \label{eq:2.22}\\
\rho(u)
& \hspace{2pt} =  \hspace{2pt}
     c_2
     \hspace{2pt}
      \frac{\left[1-\left(u^2/2i_0\right)\right]^{\frac{b}{(b-1)}}}
      {\left[1-\left(u^2/2i_0\right)\right]}; 
 \hspace{7pt}\mbox{and} \label{eq:2.23}\\      
\mu(u)
& \hspace{2pt} = \hspace{2pt}
     c_3
     \hspace{2pt}
     \left[1-\left(u^2/2i_0\right)\right]^{\frac{19}{25}};
\label{eq:2.24}
\end{align}
\noindent
where
$c_1= p_0 \hspace{2pt} T_0^{\frac{2b}{(b-1)}}$,
$c_2=c_1 \hspace{2pt}R^{-1} \hspace{2pt}T_0^{-1}$,
and
$c_3=\mu_h \hspace{2pt} T_h^{-\frac{19}{25}}\hspace{2pt} T_0^{\frac{19}{25}}$.
\end{lemma}

\begin{proof}
As previously seen in Proposition~\ref{prop1.1}, Eq. (\ref{eq:1.3}) 
and (\ref{eq:1.4}) are equivalent to Eq. (\ref{eq:1.9}). If the
boundary conditions (\ref{eq:1.12}) and (\ref{eq:1.13}) are known, 
the Lemma~\ref{lem2.1} gives $T_0 >0$ at $\Gamma_0$,
Eq. (\ref{eq:2.20}), and the expression of temperature in terms of
$u$, Eq. (\ref{eq:2.21}). Then, Eq. (\ref{eq:1.5}) provides a value
$p|_{\Gamma_0}=p_0=(n\hspace{2pt}R^*\hspace{2pt}T_0)/V>0$.
Analogously, regarding Eq. (\ref{eq:1.5}), (\ref{eq:2.21}),
and the last Remark~\ref{rem2.1}, we have that $p\not =0$ in
$\boldsymbol{\Omega_h}$. Thus, from Eq. (\ref{eq:1.7}), we get
$p_0\hspace{2pt}\left[\left(n\hspace{2pt}R^*\hspace{2pt}T_0\right)/p_0\right]^b=
p\hspace{2pt}\left[\left(n\hspace{2pt}R^*\hspace{2pt}T\right)/p\right]^b$.
This is, $p=p_0\hspace{2pt}T_0^{2b/(b-1)}\hspace{2pt}T^{b/(b-1)}$.
The substitution of Eq. (\ref{eq:2.21}) in this last expression 
conduces to (\ref{eq:2.22}). Similarly, Eq. (\ref{eq:2.22}),  
Eq. (\ref{eq:2.21}), and Eq. (\ref{eq:1.6}), which is equivalent 
to Eq. (\ref{eq:1.5}), conduces to (\ref{eq:2.23}). 
Finally, Eq. (\ref{eq:2.24}) is a consequence of Eq. (\ref{eq:1.8}),
Eq. (\ref{eq:2.21}), and the value $\mu_h$ of (\ref{eq:1.15}).
\end{proof}

\begin{remark}
Atmospheric pressure is regarded as the weight impressed by the
column of air over a point $x$ at its base 
O. G. Tietjens \cite[p. 18]{Tiet}.
Dorodnitzyn assumes $p$ to be dependent only of $x$, and that for
each  $x \in (0,L)$, $p(x,y)$ is given by its corresponding value
$p\left(x,h\left(x\right)\right)$ at $\Gamma_h$. In the
Coro\-llary~\ref{cor2.2}, we emphasize that this can be seen as a
consequence of tempera\-ture's observed linear decrease with height 
from the Earth's surface to the troposphere's 
upper border K. Saha \cite[p. 20]{Saha}. Moreover, this allows us to
consider a constant pressure value determined by the free-stream
velocity in the Theorem~\ref{theo:2.1} below.
\end{remark}

\begin{corollary}\label{cor2.2}
Under the same conditions as in Lemma~\ref{lem2.2}, let
$p(x,y)=g\int_{y}^{\infty}\rho (x,z) dz$ for all 
$(x,y)\in \boldsymbol{\Omega_h}$, where $g$ is the standard gravity
constant. If, additionally, $\beta>0$ is such that 
$T(x,y)=T_0-\beta y$ $\forall (x,y)\in 
\boldsymbol{\Omega_h}$, then for all 
$(x,y) \in \boldsymbol{\Omega_h}$:
\noindent
\begin{align}
p(x,y)
& \hspace{2pt} \cong \hspace{2pt}
c_1
\hspace{2pt}
    \left[1-\left(U^2/2i_0\right)\right]^{\frac{b}{(b-1)}};
 \hspace{7pt}\mbox{and} \label{eq:2.25}\\
\rho(x,y)
& \hspace{2pt} \cong \hspace{2pt}
      c_2
     \hspace{2pt}
      \frac{\left[1-\left(U^2/2i_0\right)\right]^{\frac{b}{(b-1)}}}
      {\left[ 1-\left( 
      u^2\left(x,y\right)/2i_0\right)\right]}.  
\label{eq:2.26}
\end{align}    
\end{corollary}

\begin{proof}
From the Lemma~\ref{lem2.2}, we have $T_0>0$, Eq. (\ref{eq:2.22}) 
and  (\ref{eq:2.23}). If $T(x,y)=T_0-\beta y$ is substituted in 
Eq. (\ref{eq:1.16}), that is equi\-valent to the given 
Eq. (\ref{eq:1.15}), and the corresponding density expression 
is used in the atmospheric pressure's definition $p(x,y)=g\int_{y}^{\infty}\rho (x,z) dz$.
Then,
    $ln\left( p\left(x,y\right)\right)-ln\left( p_0\right) =
    g\beta \left[ ln\left((T_0-\beta y)/T_0\right)\right]$.
For this reason, if $y$ is sufficiently small for the term $\beta y$
to be discarded, the variation of pre\-ssure with height may be
negligible. Hence, $p$ can be approximated by its value in each
$(x,h(x))\in \Gamma_h$. The Eq. (\ref{eq:2.22}) with values in
$\Gamma_h$ implies Eq. (\ref{eq:2.25}). Furthermore, 
Eq. (\ref{eq:2.26}) is inferred from Eq. (\ref{eq:2.25})
and (\ref{eq:2.23}).
\end{proof}

\begin{lemma}\label{lem2.3}
Let $h$ and $\boldsymbol{\Omega_h}$ be as in
Definition~\ref{defi1.2}, and $\rho$, $u$, $v$,
$T$, $p$, $\kappa$, $\mu$ 
as in Definition~\ref{defi1.2}. For each $L>0$ and
$H\colon =\max \left\{h(x)\hspace{4pt}|\hspace{4pt}x\in [0,L]\right\}$, 
there are a parameter $\epsilon \colon =H/L>0$, and a diffeomorphism
   $\phi^{\epsilon} \colon \boldsymbol{\Omega_h} \to
   \boldsymbol{\Omega_{\epsilon}}$,
   $\phi^{\epsilon}(x,y)=\left(s,\tau)
   \colon =(x/L,y/(L\epsilon)\right)$ for all $(x,y)\in \boldsymbol{\Omega_h}$.
Also, there is a vector field
   $\mathbf{v}^{\epsilon}=(u^{\epsilon},v^{\epsilon}) \in
   L^2\left(\boldsymbol{\Omega_{\epsilon}};\mathbb{R}^2\right)
   \cap L^1_{loc}\left(\boldsymbol{\Omega_{\epsilon}};\mathbb{R}^2\right)$
such that
$\mathbf{v}^{\epsilon}(s,\tau)  = 
\left( u^{\epsilon}\left(s,\tau\right),
v^{\epsilon}\left(s,\tau\right)\right)$
with
$u^{\epsilon}\left(s,\tau\right)=\left(1/L\right)u\left(Ls,L\epsilon \tau\right)$,
$v^{\epsilon}\left(s,\tau\right)=
\left(1/\left(L\epsilon\right)\right)v\left(Ls,L\epsilon \tau \right)$;
a density
$\rho^{\epsilon} \in L^1\left(\boldsymbol{\Omega_{\epsilon}};
\left(0,\infty \right)\right)$,
$\rho^{\epsilon}(s,\tau):= c_2\left[\sigma_0\right]^{b/(b-1)}
\sigma^{-1}(s,\tau)$, where $\sigma$ denotes 
$\sigma(s,\tau)=
      1-\left( 
      \left[Lu^{\epsilon}\left(s,\tau \right)\right]^2/2i_0\right)$,
$\sigma_0$ is the number $1-\left( 
      \left[LU^{\epsilon}\right]^2/2i_0\right)$,
and $U^{\epsilon}=(1/L)\hspace{2pt}U$ is the free-stream velocity 
on the curve 
$h^{\epsilon} \in C^2 \left( [0,1]\right)$ such that 
$h^{\epsilon}(x) \colon =h\left(Ls\right)/(L\epsilon)$.
Analogously, there is a dynamic viscosity 
$\mu^{\epsilon} \in L^1_{loc}\left(\boldsymbol{\Omega_{\epsilon}}\right)$
with
$\mu^{\epsilon}\colon =c_3 \sigma^{\frac{19}{25}}$.
\end{lemma}

\begin{proof}
First of all, $\phi^{\epsilon}$ is linear. Because
$Ker(\phi)=\{(0,0)\}$, it is invertible. Its Jacobian determinant is
$|D\phi^{\epsilon}|=1/(L^2 \epsilon) >0$. Consequently, 
by the Inverse Function Theorem, $\phi^{\epsilon}$ is a
diffeomorphism of $\boldsymbol{\Omega_{h}}$. Second, the vector
field is obtained via the Chain Rule: Let $t\in [0,\infty)$ be 
the time, then 
$u^{\epsilon}=\partial s/\partial t=(\partial s/\partial x)(\partial x/\partial t) =(1/L)u$. 
Similarly, we obtain $v^{\epsilon}$ and the free-stream velocity
$U^{\epsilon}$. Moreover, if 
$u \in L^2\left(\boldsymbol{\Omega_h}\right)$,
\begin{eqnarray*}
   \Vert u\Vert^2_{L^2\left(\boldsymbol{\Omega_h}\right)} &=&
   \iint_{\boldsymbol{\Omega_h}}u^2(x,y)\hspace{2pt}dx\hspace{2pt}dy =
   L^2\epsilon \iint_{\boldsymbol{\Omega_{\epsilon}}}[Lu^{\epsilon}]^2(s,\tau)
   \hspace{2pt}ds\hspace{2pt}d\tau.                              
\end{eqnarray*}
So that,
\begin{equation}\label{eq:2.27}
    \Vert u^{\epsilon}\Vert^2_{L^2\left(\boldsymbol{\Omega_{\epsilon}}\right)}
    =
    L^4\epsilon \Vert u\Vert^2_{L^2\left(\boldsymbol{\Omega_h}\right)}< \infty,
\end{equation}
and $u^{\epsilon} \in L^2\left(\boldsymbol{\Omega_{\epsilon}}\right)$.
In the same way, $u^{\epsilon} \in
L^1_{loc}\left(\boldsymbol{\Omega_{\epsilon}}\right)$, 
and $v^{\epsilon} \in L^1_{loc}\left(\boldsymbol{\Omega_{\epsilon}}\right)
\cap L^2\left(\boldsymbol{\Omega_{\epsilon}}\right)$.
Finally, the density $\rho^{\epsilon}$, the curve $h^{\epsilon}$, 
and the dynamic viscosity $\mu^{\epsilon}$ 
are determined by the corresponding commutative diagrams with $\phi^{\epsilon}$.
\end{proof}

\begin{theorem}[\textbf{Adimensional Model}]\label{theo:2.1}
Let $\rho, u,v,T,p,\kappa,\mu$ be as in Definition~\ref{defi1.2}. 
Suppose they satisfy the Dorodnitzyn's Boundary Layer Model 
given by equations (\ref{eq:1.1}), (\ref{eq:1.2}), (\ref{eq:1.3}),
(\ref{eq:1.4}), (\ref{eq:1.5}), (\ref{eq:1.7}), (\ref{eq:1.8}) 
with boundary conditions 
(\ref{eq:1.12}), (\ref{eq:1.13}), (\ref{eq:1.14}), (\ref{eq:1.15}),
(\ref{eq:1.16}), (\ref{eq:1.17}).
Additio\-nally, assume 
$p=
c_1
\hspace{2pt}
    \left[1-\left(U^2/2i_0\right)\right]^{\frac{b}{(b-1)}}$ in $\boldsymbol{\Omega_h}$.
Then, $u^{\epsilon}$, and $v^{\epsilon}$, as defined in the 
Lemma~\ref{lem2.3} above, verify the following system in
$\boldsymbol{\Omega_{\epsilon}}$:
\noindent
  \begin{align}
div\left(\rho^{\epsilon} u^{\epsilon},\rho^{\epsilon} v^{\epsilon}\right)
& \hspace{2pt} = \hspace{2pt}
0;
 \hspace{7pt}\mbox{and} \label{eq:2.28}\\
L^2 \epsilon^2
  \rho^{\epsilon} 
  \left( 
   u^{\epsilon} \frac{\partial u^{\epsilon}}{\partial s}
   \hspace{3pt} + \hspace{3pt}
  v^{\epsilon} \frac{\partial u^{\epsilon}}{\partial \tau}
  \right)
& \hspace{2pt} = \hspace{2pt}
c_3
\hspace{2pt}
      \frac{\partial}{\partial \tau} 
  \left[
 \sigma^{\frac{19}{25}}
      \frac{\partial u^{\epsilon}}{\partial \tau}
      \right],  
\label{eq:2.29}
\end{align}
with boundary conditions:
\noindent
\begin{align}
\left(u^{\epsilon},v^{\epsilon}\right)|_{\phi^{\epsilon}\left(\Gamma_0\right)}
& \hspace{2pt} = \hspace{2pt}
(0,0); \label{eq:2.30}\\
\left(u^{\epsilon},v^{\epsilon}\right)|_{\phi^{\epsilon}\left(\Gamma_h\right)}
& \hspace{2pt} = \hspace{2pt}
\left(-LU^{\epsilon},0\right);
 \hspace{7pt}\mbox{and} \label{eq:2.31}\\
 \left(u^{\epsilon}\left(0,\tau \right),0\right)
& \hspace{2pt} = \hspace{2pt}
\left(u^{\epsilon}\left(1,\tau \right),0\right),
   \hspace{4pt} \forall \tau \in \left[0,\delta/(L\epsilon)\right]; 
\label{eq:2.32}
\end{align}  
where $\rho^{\epsilon}$ and $\sigma$ depend of $u^{\epsilon}$,
in the way described in Lemma~\ref{lem2.3}.
\end{theorem}

\begin{proof}
Considering the new directions, the generalized partial derivatives 
$\partial u/\partial x=\partial u^{\epsilon}/\partial s$;
$\partial u/\partial y  =  (1/\epsilon)\hspace{2pt} 
\partial u^{\epsilon}/\partial \tau$; \hspace{2pt}
$\partial v/\partial y=\partial v^{\epsilon}/\partial \tau$;\hspace{2pt}
$u\hspace{2pt}(\partial u/\partial x) = L\hspace{2pt} u^{\epsilon}\hspace{2pt}
(\partial u^{\epsilon}/\partial s)$; \hspace{2pt} and
$v(\partial u/\partial y)  =  L\hspace{2pt} v^{\epsilon}
\hspace{2pt}(\partial u^{\epsilon}/\partial \tau)$.
The weak derivative
$\partial p/\partial x=0$ because $U$ is constant. 
Similarly, $p$ allows us to see  
\[\rho=
      c_2
     \hspace{2pt}
      \left[1-\left(U^2/2i_0\right)\right]^{\frac{b}{(b-1)}}/
      \left[ 1-\left( 
      u^2\left(x,y\right)/2i_0\right)\right].\]
Therefore,      
$\partial/\partial y \left[ \mu
      \left(
          \partial u/\partial y
      \right)
      \right]=$
      $L^{-1}\epsilon^{-2} c_3
      \hspace{2pt}
    \partial/\partial \tau
  \left[
 \sigma^{\frac{19}{25}}
      \left(
      \partial u^{\epsilon}/\partial \tau
      \right)
      \right]$.
Finally, each term is substituted on each side of Eq. (\ref{eq:1.1})
and (\ref{eq:1.2}) to obtain Eq. (\ref{eq:2.28}) and 
(\ref{eq:2.29}).
\end{proof}

\subsection{Incompressible Model}

The domain's shape $\boldsymbol{\Omega_h}$ described in
Definition~\ref{defi1.1} is different from the rectangular one 
in the original Dorodnitzyn's article. In addi\-tion, there is no
domain $\boldsymbol{\Omega_{\epsilon}}$ in Dorodnitzyn's work,
because this was obtained with the appli\-cation of Bayada and
Chambat's diffeo\-morphism $\phi^{\epsilon}$. Therefore, it is
necessary to make an adjustment on Dorodnitzyn's change of
variables to take into account the points 
$(s,\tau)\in \boldsymbol{\Omega_{\epsilon}}$  over a height 
$\phi^{\epsilon}\left(0,\delta\right)=\left(0,\delta/(L\epsilon)
\right)$,
as is done in Eq. (\ref{eq:2.34}) below. This new diffeormorphism
allows us to take the Adimensional Model into an incompressible
form.

\begin{lemma}\label{lem2.4}
Let $h \in C^{2}\left([0,L],\left(0,\infty\right)\right)$ 
have only one critical point which is a maximum.
Let $\boldsymbol{\Omega_{\epsilon}}$ and $\rho^{\epsilon}$ be 
as described in Lemma~\ref{lem2.3}.
Suppose that the weak derivative
$\partial u/\partial x=0$ \emph{a.e.} in $\boldsymbol{\Omega_{h}}$.
Then, there is a diffeomorphism
$\boldsymbol{\eta} =(\eta_1,\eta_2)\colon \boldsymbol{\Omega_{\epsilon}}\to \mathbb{R}^2$
 such that $\forall s \in [0,1]$: 
\noindent
\begin{align}
  \eta_1 (s,\tau)
& \hspace{2pt} = \hspace{2pt}
  \int_{0}^{s} \frac{1 }
  {\rho^{\epsilon}\left(\zeta,\tau \right)}
  \hspace{2pt}d\zeta,
 \hspace{7pt}
\forall \tau \in \left[0,\delta/(L\epsilon)\right); \label{eq:2.33}\\
  \eta_1 (s,\tau)
& \hspace{2pt} = \hspace{2pt}
  \int_{\tilde{s}}^{s} \frac{1 }
  {\rho^{\epsilon}\left(\zeta,\tau \right)}
  \hspace{2pt}d\zeta,
  \hspace{9pt}
  \forall \tau \in h^{\epsilon}\left(\left[0,1\right]\right);
 \hspace{7pt}\mbox{and} \label{eq:2.34}\\
 \eta_2(s,\tau)
& \hspace{2pt} = \hspace{2pt}
  \int_0^{\tau} \rho^{\epsilon}\left(s,\xi \right)
  \hspace{2pt}d\xi;  
\label{eq:2.35}
\end{align}  
\noindent
where $\tilde{s}$ is the preimage of 
$\tau=h^{\epsilon}\left(\tilde{s}\right)\in h^{\epsilon}\left(\left[0,1\right]\right)$
such that the slope $\partial h^{\epsilon}/\partial s \left(\tilde{s}\right)\geq 0$.
\end{lemma}

\begin{proof}
By definition, $\forall (s,\tau)\in \boldsymbol{\Omega_{\epsilon}}$, $\rho^{\epsilon}(s,\tau)= \rho(Ls,L\epsilon \tau)>0$.
From the Remark~\ref{rem2.1}, we know that $\sigma$ is positive
and bounded by $1$. In addition, if $h$ has one unique critical
maximum in its domain, $h^{\epsilon}$ does as well.
In fact, the top cover of $\boldsymbol{\Omega_{\epsilon}}$ is given
by the curve $h^{\epsilon}$, where each image $\tau=h^{\epsilon}(s)$,
different from its cusp, has exactly two preimages,
one of them on the ascending part of the curve where 
$\partial h^{\epsilon}/\partial s \left(\tilde{s}\right)\geq 0$.
So that the horizontal segment $(\tilde{s},s)\times \{\tau\}$
is contained in $\boldsymbol{\Omega_{\epsilon}}$.
Thus, each Riemann integral $\eta_1 (s,\tau)$ 
is the limit of an of increasing and bounded sequence of Darboux 
sums which add positive values taken by the function $\sigma$ over
a horizontal and bounded segment contained in $\boldsymbol{\Omega_{\epsilon}}$.
As a result, for each $(s,\tau)\in \boldsymbol{\Omega_{\epsilon}}$, 
the sequence of sums converges and $\eta_1$ is well defined.
In addition, the Remark~\ref{rem2.1} implies that $\sigma$ is strictly positive. 
Then, $\eta_2$ is a well defined function in $\boldsymbol{\Omega_{\epsilon}}$.
Two of its partial derivatives are
$\partial \eta_1/\partial s=1/\rho^{\epsilon}$, and 
$\partial \eta_2/\partial \tau=\rho^{\epsilon}$.
By the Monotone Convergence Theorem, if 
$\left(\partial u/\partial x\right)=0$ \emph{a. e.} in $\boldsymbol{\Omega_h}$,
we calculate the product
$\left(\partial \eta_1/\partial \tau\right)\left(\partial \eta_2/\partial s\right)=0$.
Then, the Jacobian determinant $\left|D\boldsymbol{\eta}\right|=1$.
Hence, by the Inverse Function Theorem,
$\boldsymbol{\eta}$ is a diffeormorphism of $\boldsymbol{\Omega_{\epsilon}}$.
\end{proof}

\begin{theorem}[\textbf{Incompressible Model}]\label{theo:2.2}
Let $\rho, u,v,T,p,\kappa,\mu$ be as in Definition~\ref{defi1.2}.
Suppose they satisfy the Dorodnitzyn's Boundary Layer Model given by
equations (\ref{eq:1.1}), (\ref{eq:1.2}), (\ref{eq:1.3}),
(\ref{eq:1.4}), (\ref{eq:1.5}), (\ref{eq:1.7}), (\ref{eq:1.8}) 
with boundary conditions 
(\ref{eq:1.12}), (\ref{eq:1.13}), (\ref{eq:1.14}), (\ref{eq:1.15}),
(\ref{eq:1.16}), (\ref{eq:1.17}),
$p(x,y)=p\left(x, h\left(x\right)\right)$ $\forall (x,y) \in \boldsymbol{\Omega_h}$,
and 
$\partial u/\partial x=0$ \emph{a.e.} in $\boldsymbol{\Omega_{h}}$.
Consider $u^{\epsilon}$, $v^{\epsilon}$, $\rho^{\epsilon}$,
$\sigma$, and $\sigma_0$ as in  Lemma~\ref{lem2.3}, and the domain
$\boldsymbol{\eta}\left(\boldsymbol{\Omega_{\epsilon}}\right)=\boldsymbol{\Omega}$
as defined in Lemma~\ref{lem2.4}.
Then, there exists a \emph{stream-function} $\psi$ such that
$\partial \psi/\partial s=-\hspace{2pt}\rho^{\epsilon}\hspace{2pt} v^{\epsilon}$,
$\partial \psi/\partial \tau=\hspace{2pt}\rho^{\epsilon}\hspace{2pt} u^{\epsilon}$;
and a vector field $F^{\epsilon}=(F_1^{\epsilon},F_2^{\epsilon})\in
L^2\left(\boldsymbol{\Omega}; \mathbb{R}^2\right)\cap 
L^1_{loc}\left(\boldsymbol{\Omega}; \mathbb{R}^2\right)$,
$F_1^{\epsilon} =\partial \psi/\partial \eta_2$
and $F_2^{\epsilon}=-\partial \psi/\partial \eta_1$,
that satisfies:
\begin{align}
div\hspace{2pt}(F_1^{\epsilon},F_2^{\epsilon})
& \hspace{2pt} = \hspace{2pt}
0
 \hspace{7pt}\mbox{and} \label{eq:2.36}\\
L^2\epsilon^2 \left\{ 
F_1^{\epsilon}\frac{\partial F_1^{\epsilon}}{\partial \eta_1}+
F_2^{\epsilon}\frac{\partial F_1^{\epsilon}}{\partial \eta_2}
\right\}
& \hspace{2pt} = \hspace{2pt}
C \tilde{\sigma}^{-1} \frac{\partial}{\partial \eta_2}
\left[\tilde{\sigma}^{-\frac{6}{25}} 
\frac{\partial F_1^{\epsilon}}{\partial \eta_2}\right], 
\label{eq:2.37}
\end{align}  
\noindent
where
$\boldsymbol{\eta^{-1}}$ is the inverse function of $\boldsymbol{\eta}$,
$\tilde{\sigma}=\sigma \circ \boldsymbol{\eta^{-1}}$, and
$C\hspace{2pt}=\hspace{2pt}c_3 \hspace{2pt}c_2^2\hspace{2pt}
\sigma_0^{\frac{2b}{b-1}}$ as denoted in Lemma~\ref{lem2.2}.
Moreover, the boundary conditions are given, for all 
$\left(\eta_1,\eta_2\right)\in \partial \boldsymbol{\Omega}$, by:
\begin{eqnarray}\label{eq:2.38}
F^{\epsilon}|_{\partial
\boldsymbol{\Omega}}
\left(\eta_1\left(s,\tau \right),\eta_2\left(s,\tau \right)\right)
&=&
\left(u^{\epsilon}|_{\partial
\boldsymbol{\Omega_{\epsilon}}}\left(s,\tau \right),0\right).
\end{eqnarray}
\end{theorem}

\begin{proof}
First, under these conditions, $u^{\epsilon}$ and $v^{\epsilon}$ 
verify the system described in Theorem~\ref{theo:2.1},
and, according to Lemma~\ref{lem2.4}, $\boldsymbol{\eta}$ is a diffeomorphism of $\boldsymbol{\Omega_{\epsilon}}$.
Second, Eq. (\ref{eq:2.36}) allows the definition of a 
stream-function given a fixed point $(s_0,\tau_0)\in \boldsymbol{\Omega_{\epsilon}}$. 
Third, Eq. (\ref{eq:2.37}) is wri\-tten in terms of its partial
derivatives. Then, these partial derivatives are calculated in 
the new coordinates $\eta_1$ and $\eta_2$.
Finally, the left side and right side of the new equation for the
stream-function's original partial derivatives is presented in the 
new directions, and substituted by the field's $F^{\epsilon}$
coordinate functions. The boundary conditions are determined
as a direct result of the vector field's  definition,
where it can be seen that it satisfies the relations:
For all $(\eta_1,\eta_2)\in \boldsymbol{\Omega}$
such that $\boldsymbol{\eta}(s,\tau)=(\eta_1,\eta_2)$,
\begin{align}
F_1^{\epsilon}(\eta_1,\eta_2)
& \hspace{2pt} = \hspace{2pt}
\frac{\partial \psi}{\partial \eta_2}(\eta_1,\eta_2)
=\frac{1}{\rho^{\epsilon}}\frac{\partial \psi}{\partial \tau}(s,\tau)=u^{\epsilon}(s,\tau);
 \hspace{7pt}\mbox{and} \label{eq:2.39}\\
F_2^{\epsilon}(\eta_1,\eta_2)
& \hspace{2pt} = \hspace{2pt}
-\frac{\partial \psi}{\partial \eta_1}(\eta_1,\eta_2)=
-\rho^{\epsilon} \frac{\partial \psi}{\partial s}(s,\tau)
=-(\rho^{\epsilon})^2 v^{\epsilon}(s,\tau).
\label{eq:2.40}
\end{align}
\noindent
In particular, the repeated argument made for Eq. (\ref{eq:2.26})
and the Cauchy-Schwarz inequality for the $L^2$-norm implies that
the vector $F^{\epsilon}\in
L^2\left(\boldsymbol{\Omega}; \mathbb{R}^2\right)$.
Moreover, $F^{\epsilon}\in L^1_{loc}\left(\boldsymbol{\Omega}; \mathbb{R}^2\right)$
and has inherited locally integrable weak partial derivatives.

If Eq. (\ref{eq:2.39}), for each fixed point $(s_0,\tau_0)\in \boldsymbol{\Omega_{\epsilon}}$
and each $\left(s,\tau\right)\in \boldsymbol{\Omega}_{\epsilon}$,
the Poincare's Lemma implies that the integral
\begin{equation*}
\psi \left(s,\tau\right)\colon =\int_{\gamma}
\left(-\rho^{\epsilon}v^{\epsilon}\right)ds+
\left(\rho^{\epsilon}u^{\epsilon}\right)d\tau,
\end{equation*}
has the same real value for every $\gamma \colon [0,1] \to \boldsymbol{\bar{\Omega}_{\epsilon}}$ 
such that $\gamma(0)=(s_0,\tau_0)$ and $\gamma(1)=(s,\tau)$.
This is, the streamfunction $\psi$ is well defined on $\boldsymbol{\Omega}_{\epsilon}$.

In order to calculate its derivatives, it is enough to pick a
trajectory built by pieces where one variable is fixed. 
Substitution of $u^{\epsilon}$ and $v^{\epsilon}$ in terms of 
the streamfunction's derivatives,
$\partial \psi/\partial s=-\hspace{2pt}\rho^{\epsilon}
\hspace{2pt} v^{\epsilon}$ and
$\partial \psi/\partial \tau=\hspace{2pt}\rho^{\epsilon}
\hspace{2pt} u^{\epsilon}$, and the hypothesis that
$\rho^{\epsilon}$ is not null at any point of its domain,
allows us to write Eq. (\ref{eq:2.37}) in terms of  $\psi$ as: 
\begin{equation*}
  L^2 \epsilon^2
    \left[
    \frac{\partial \psi}{\partial \tau}
    \frac{\partial }{\partial s}
    \left[
    \frac{1}{\rho^{\epsilon}}
    \frac{\partial \psi}{\partial \tau}
    \right]
    -
    \frac{\partial \psi}{\partial s}
    \frac{\partial }{\partial \tau}
    \left[
    \frac{1}{\rho^{\epsilon}}
    \frac{\partial \psi}{\partial \tau}
    \right]
    \right]
   = 
c_3
      \frac{\partial}{\partial \tau} 
  \left[
      \sigma^{\frac{19}{25}}
      \frac{\partial }{\partial \tau}
      \left(
      \frac{1}{\rho^{\epsilon}}
      \frac{\partial \psi}{\partial \tau}
      \right)
      \right].
\end{equation*}
\noindent

Additionally, there is a new domain 
$\boldsymbol{\Omega}\subset \mathbb{R}^2$ where:
\begin{equation}\label{eq:2.41}
\frac{\partial \psi}{\partial \tau}=
\frac{\partial \psi}{\partial \eta_2}
\frac{\partial \eta_2}{\partial \tau}=
\rho^{\epsilon} 
\frac{\partial \psi}{\partial \eta_2}
\hspace{7pt}\&\hspace{7pt}
\frac{\partial \psi}{\partial s}=
\frac{\partial \psi}{\partial \eta_1}
\frac{\partial \eta_1}{\partial s}
=
\frac{1}{\rho^{\epsilon} }
\frac{\partial \psi}{\partial \eta_1}.
\end{equation}
Once again, substitution of identities in Eq. (\ref{eq:2.41})
in the left side of the equation above and the definition of
$F^{\epsilon}$ give a new expression for the nonlinear term as:
\begin{eqnarray*}
  L^2 \epsilon^2
    \left[
    \frac{\partial \psi}{\partial \eta_2}
    \frac{\partial^2 \psi}{\partial \eta_1 \partial \eta_2}
    -
    \frac{\partial \psi}{\partial \eta_1}
    \frac{\partial^2 \psi}{\partial \eta_2^2}
    \right]
        &=&
L^2 \epsilon^2
    \left[
F_1^{\epsilon}
    \frac{\partial F_1^{\epsilon}}{\partial \eta_1}
    +
F_2^{\epsilon}
    \frac{\partial F_1^{\epsilon}}{\partial \eta_2}
    \right].
\end{eqnarray*}
\noindent
Similarly, by the second identity in Eq. (\ref{eq:2.41})
and the definition of $F_1^{\epsilon}$, the right side of the same
equation is:
\begin{eqnarray*}
       c_3\hspace{2pt}
      \frac{\partial}{\partial \tau} 
  \left[
      \sigma^{\frac{19}{25}}
      \frac{\partial }{\partial \tau}
      \left(
      \frac{1}{\rho^{\epsilon}}
      \frac{\partial \psi}{\partial \tau}
      \right)
      \right]
      &=&
      c_3 \hspace{2pt}
    \frac{\partial \eta_2}{\partial \tau}
    \hspace{2pt}
    \frac{\partial}{\partial \eta_2} 
  \left[
      \sigma^{\frac{19}{25}}
      \hspace{2pt}
      \rho^{\epsilon}
      \hspace{2pt}
\frac{\partial^2 \psi}{\partial \eta_2^2}
      \right],
       \\
    &=&
    c_3\hspace{2pt}
\rho^{\epsilon}\hspace{2pt}
    \frac{\partial}{\partial \eta_2} 
  \left[
      \sigma^{\frac{19}{25}}
      \hspace{2pt}
\rho^{\epsilon}
             \hspace{2pt}
\frac{\partial^2 \psi}{\partial \eta_2^2}
      \right], \\
      &=&
      c_3\hspace{2pt}
c_1^2
\hspace{2pt}
\sigma_0^{\frac{2b}{(b-1)}}
\hspace{2pt}
\tilde{\sigma}^{-1}
\frac{\partial}{\partial \eta_2}
\left[\sigma^{\frac{19}{25}-1}
\hspace{2pt}
\frac{\partial F_1^{\epsilon}}{\partial \eta_2}\right].     
\end{eqnarray*}
Therefore, the vector field $F^{\epsilon}\in
L^2\left(\boldsymbol{\Omega}; \mathbb{R}^2\right)\cap
L^1_{loc}\left(\boldsymbol{\Omega}; \mathbb{R}^2\right)$,
and its locally integrable weak partial derivatives, satisfy the
incompressible system of Eq. (\ref{eq:2.36}) and 
(\ref{eq:2.37}) with boundary conditions given 
by Eq. (\ref{eq:2.38}).
\end{proof}

\subsection{Dorodnitzyn Boundary Layer Limit Formula}
\label{subsec:2.3}

Alberto Bressan's \cite{Bressan13} book 
\emph{Lecture Notes on Functional Analysis: 
With Applications to Linear Partial Differential Equations} 
provides an excellent account of Sobolev 
Embeedding Theorems, as they will be used in this section.

\begin{theorem}\label{theo:2.3}
Under the same hypothesis of Theorem~\ref{theo:2.2}, there is an
estimate: 
\begin{eqnarray}\label{eq:2.42}
 \Vert \nabla F^{\epsilon} 
 \Vert_{L^2\left(\boldsymbol{\Omega};\mathbb{R}^2\right)}
\leq
\frac{c_2\hspace{2pt}U^3}{2\hspace{2pt}C}
\end{eqnarray}
\end{theorem}

\begin{proof}
From Theorem~\ref{theo:2.2}, the vector field 
$F^{\epsilon}=(F_1^{\epsilon},F_2^{\epsilon})$
verifies the system of Eq. (\ref{eq:2.36}) and (\ref{eq:2.37}) in
$\boldsymbol{\Omega}$ 
with boundary conditions determined by (\ref{eq:2.38}).
In particular, there is an underlying assumption that the Laplacian
    \begin{eqnarray}\label{eq:2.43}
    \Delta F_2^{\epsilon} &=&
    \sum_{i=1,2} 
\frac{\partial^2 F_2^{\epsilon}}{\partial \eta_i^2}=0,
    \end{eqnarray}
because the conservation of momentum equation for $F_2^{\epsilon}$ 
is considered null. Furthermore, from $\partial u/\partial x=0$
\emph{a.e.} in $\boldsymbol{\Omega_h}$ and Eq. (\ref{eq:2.39}),
it can be seen that:
 \begin{eqnarray}\label{eq:2.44}
\frac{\partial^2 F_1^{\epsilon}}{\partial \eta_1^2}&=&0.
    \end{eqnarray}
    
Let $\mathcal{F}_1 \left(F^{\epsilon}\right)$ denote
the inner product in
$L^2\left(\boldsymbol{\Omega};\mathbb{R}^2\right)$ of 
$F^{\epsilon}$ and the vector
$\left(F^{\epsilon}\cdot \nabla \right) F^{\epsilon}=
\left(\sum_{i=1,2}F_i^{\epsilon}\hspace{2pt} 
\frac{\partial F_j^{\epsilon}}{\partial \eta_i}\right)_{j=1,2}$
in the space $L^2\left(\boldsymbol{\Omega};\mathbb{R}^2\right)$.
Namely,
\begin{eqnarray*}
\mathcal{F}_1 \left(F^{\epsilon} \right) 
&=&
\frac{1}{2}
\int \!\!\! \int_{\boldsymbol{\Omega}}   
\sum_{i=1,2} F_i^{\epsilon} 
\left(\sum_{j=1,2} \frac{\partial \left(F_j^{\epsilon}\right)^2}{\partial \eta_i}\right) 
\hspace{2pt} d\boldsymbol{\eta}.
\end{eqnarray*}

From Eq. (\ref{eq:2.39}) and the boundary conditions 
(\ref{eq:2.30}), (\ref{eq:2.31}), and (\ref{eq:2.32}) 
for $v^{\epsilon}$, we have:
\begin{eqnarray}\label{eq:2.45}
    F_2^{\epsilon}|_{\partial \boldsymbol{\Omega}}&=&0.
    \end{eqnarray}
If $div\hspace{2pt}(F_1^{\epsilon},F_2^{\epsilon})=0$,
by the Gauss-Ostrogradsky Theorem and Eq. (\ref{eq:2.45}):
\begin{eqnarray*}
\mathcal{F}_1 \left(F^{\epsilon}\right)
&=&
-\frac{1}{2}
\int \!\!\! \int_{\boldsymbol{\Omega}}   
\left\{\left(F_1^{\epsilon} \right)^2 + \left(F_1^{\epsilon} \right)^2\right\}
div\left(F_1^{\epsilon},F_2^{\epsilon}\right)
\hspace{2pt} d\boldsymbol{\eta}\\
&  &
+\frac{1}{2}
\int_{\partial \boldsymbol{\Omega}}   
\left(\left(F_1^{\epsilon}\right)^2,0\right)
\hspace{2pt}\cdot \hspace{2pt}
\mathbf{n}
\hspace{2pt} dS,
\\
&=&
\frac{1}{2}
\int_{\partial \boldsymbol{\Omega}}   
\left(\left(F_1^{\epsilon}\right)^3,0\right)
\hspace{2pt}\cdot \hspace{2pt}
\mathbf{n}
\hspace{2pt} dS,
\end{eqnarray*}
where $\mathbf{n}$ is the outward pointing unitary normal 
vector field of the topolo\-gical boundary 
$\partial \boldsymbol{\Omega}$. 
Because $\boldsymbol{\eta}$ is a diffeomorphism,
$\boldsymbol{\eta}\left(
\partial \boldsymbol{\Omega_{\epsilon}}\right)
=
\partial \boldsymbol{\Omega}$.
This is, $\partial \boldsymbol{\Omega}=
\boldsymbol{\eta}\left(\phi^{\epsilon}\left(\Gamma_0\right)\right)
\cup
\boldsymbol{\eta}\left(\phi^{\epsilon}\left(\Lambda_0\right)\right)
\cup
\boldsymbol{\eta}\left(\phi^{\epsilon}\left(\Lambda_L\right)\right)
\cup
\boldsymbol{\eta}\left(\phi^{\epsilon}\left(\Gamma_h\right)\right)$.
The no-slip boundary condition for $F_1^{\epsilon}$ in
$\boldsymbol{\eta}\left(\phi^{\epsilon}\left(\Gamma_0\right)\right)$
is inherited from $u^{\epsilon}$ by Eq. (\ref{eq:2.30}). This way,
\begin{eqnarray*}
  \int_{\boldsymbol{\eta}\left(\phi^{\epsilon}\left(\Gamma_0\right)\right)}   
\left(\left(F_1^{\epsilon}\right)^3,0\right)
\hspace{2pt}\cdot \hspace{2pt}
\mathbf{n}
\hspace{2pt} dS
&=&
0.
\end{eqnarray*}
The periodic boundary conditions of $u^{\epsilon}$ establi\-shed 
in Eq. (\ref{eq:2.32}) imply that 
$\forall \tau \in [0,\delta/L\epsilon]$, or for all
$(1,\tau) \in \boldsymbol{\eta}\left(\phi^{\epsilon}
\left(\Lambda_L\right)\right)$:
\begin{eqnarray*}
\eta_1\left(1,\tau\right)
&=&
c_2^{-1}\sigma_0^{-b/(b-1)}\left(1-
\left[\left[u^{\epsilon}\left(1,\tau\right)\right]^2-
\left[u^{\epsilon}\left(0,\tau\right)\right]^2\right]
\right)\\
&=& c_2^{-1}\sigma_0^{-b/(b-1)}.
\end{eqnarray*}
In addition, $\eta_1(0,\tau)=0$ 
$\forall \tau \in [0,\delta/L\epsilon]$,
\emph{i.e} $\forall (0,\tau) \in
\boldsymbol{\eta}\left(\phi^{\epsilon}\left(\Lambda_0\right)
\right)$. 
Thus, the partial derivatives 
$\partial \eta_1/\partial \tau \hspace{2pt}(0,\tau)=
\partial \eta_1/\partial \tau \hspace{2pt}(1,\tau)=0$ 
$\forall \tau \in [0,\delta/L\epsilon]$,
and the boundary's sections
$\boldsymbol{\eta}\left(\phi^{\epsilon}\left(\Lambda_0\right)
\right)$
and
$\boldsymbol{\eta}\left(\phi^{\epsilon}\left(\Lambda_L\right)
\right)$
are vertical.
Consequently, Eq. (\ref{eq:2.39}) implies that:
\begin{eqnarray*}
  \int_{\boldsymbol{\eta}\left(\phi^{\epsilon}\left(\Lambda_0\right)\right)}   
\left(\left(F_1^{\epsilon}\right)^3,0\right)
\hspace{2pt}\cdot \hspace{2pt}
\mathbf{n}
\hspace{2pt} dS
&=&
-\int_{0}^{\frac{\delta}{L\epsilon}}   
\left[u^{\epsilon}\left(0,\tau\right) \right]^3 \frac{\partial \eta_1}{\partial \tau}\left(0,\tau\right)
\hspace{2pt} d\tau
=0.
\end{eqnarray*}
Similarly,
\begin{eqnarray*}
\int_{\boldsymbol{\eta}\left(\phi^{\epsilon}\left(\Lambda_L\right)\right)}   
\left(\left(F_1^{\epsilon}\right)^3,0\right)
\hspace{2pt}\cdot \hspace{2pt}
\mathbf{n}
\hspace{2pt} dS
&=&
0.
\end{eqnarray*}
As a result,
the product $\mathcal{F}_1 \left(F^{\epsilon}\right)$
is determined only by the free-stream velocity:
\begin{eqnarray}\label{eq:2.46}
  \mathcal{F}_1 \left(F^{\epsilon}\right)
&=&
\frac{1}{2}
\hspace{2pt}
\int_{\boldsymbol{\eta}\left(\phi^{\epsilon}\left(\Gamma_h\right)
\right)}   
\left([-LU^{\epsilon}]^3,0\right)
\hspace{2pt}\cdot \hspace{2pt}
\mathbf{n}
\hspace{2pt} dS.
\end{eqnarray}

Let $\mathcal{F}_2 \left(F^{\epsilon}\right)$ designate the product
of $F^{\epsilon}$ and the vector corres\-ponding to the right side
of Eq. (\ref{eq:2.37}) in the space 
$L^2\left(\boldsymbol{\Omega};\mathbb{R}^2\right)$:
\begin{eqnarray*}
\mathcal{F}_2 \left(F^{\epsilon}\right)
&=&
\int \!\!\! \int_{\boldsymbol{\Omega}}
\left(F_1^{\epsilon},F_2^{\epsilon}\right)
\cdot
\left(
C
\tilde{\sigma}^{-1}
\frac{\partial}{\partial \eta_2}
\left[\tilde{\sigma}^{-\frac{6}{25}}
\frac{\partial F_1^{\epsilon}}{\partial \eta_2}\right],
\hspace{2pt}
0\right)
\hspace{2pt}
d\boldsymbol{\eta},
\\
&=&
C
\int \!\!\! \int_{\boldsymbol{\Omega}}
F_1^{\epsilon}
\hspace{2pt}
\tilde{\sigma}^{-1}
\frac{\partial}{\partial \eta_2}
\left[\tilde{\sigma}^{-\frac{6}{25}}
\frac{\partial F_1^{\epsilon}}{\partial \eta_2}\right]
\hspace{2pt}
d\boldsymbol{\eta}.
\end{eqnarray*}
In fact, $\tilde{\sigma}^{-1}>1$ in $\boldsymbol{\bar{\Omega}}$.
Then, by the Gauss-Ostrogradsky Theorem, Eq. (\ref{eq:2.39}), 
and the boundary conditions (\ref{eq:2.30}), (\ref{eq:2.31}) 
and (\ref{eq:2.32}), we have:
\begin{eqnarray*}
\mathcal{F}_2 \left(F^{\epsilon}\right)
&\geq&
C \hspace{2pt} 
\int \!\!\! \int_{\boldsymbol{\Omega}}
F_1^{\epsilon}
\hspace{2pt}
\frac{\partial}{\partial \eta_2}
\left[\tilde{\sigma}^{-\frac{6}{25}}
\frac{\partial F_1^{\epsilon}}{\partial \eta_2}\right]
\hspace{2pt}
d\boldsymbol{\eta},\\
&\geq&
-C 
\left[
\int \!\!\! \int_{\boldsymbol{\Omega}}
\left(
\frac{\partial F_1^{\epsilon}}{\partial \eta_2}\right)^2
\hspace{2pt}
d\boldsymbol{\eta}
+
\frac{1}{2}
\hspace{2pt}
\int_{\partial \boldsymbol{\Omega}} 
\left(0,
\hspace{2pt}
\frac{\partial \left(F_1^{\epsilon}\right)^2}{\partial \eta_2}
\right)
\cdot
\boldsymbol{n}
\hspace{2pt} dS
\right] \\
&=&
-C 
\hspace{2pt}
\int \!\!\! \int_{\boldsymbol{\Omega}}
\left(
\frac{\partial F_1^{\epsilon}}{\partial \eta_2}\right)^2
\hspace{2pt}
d\boldsymbol{\eta}.
\end{eqnarray*}
This is because the restriction of $F_1^{\epsilon}$ to 
$\boldsymbol{\eta}\left(\phi^{\epsilon}\left(\Gamma_h\right)\right)$
is constant, the deri\-vative $\partial F_1^{\epsilon}/\partial
\eta_2|_{\boldsymbol{\eta}\left(\phi^{\epsilon}\left(\Gamma_h\right)
\right)}=0$,
and the periodic boundary condition (\ref{eq:2.32}) makes vertical
the sections $\boldsymbol{\eta}\left(\phi^{\epsilon}\left(\Lambda_0
\right)\right)$
and $\boldsymbol{\eta}\left(\phi^{\epsilon}\left(\Lambda_L\right)
\right)$,
so that the normal 
$\mathbf{n}|_{\boldsymbol{\eta}\left(\phi^{\epsilon}\left(\Lambda_0
\right)\right)
\cup 
\boldsymbol{\eta}\left(\phi^{\epsilon}\left(\Lambda_L\right)\right)}=(\pm 1,0)$.

In similar fashion, given Eq. (\ref{eq:2.44}) and  
$\partial F_1^{\epsilon}/\partial 
\eta_1|_{\boldsymbol{\eta}\left(\phi^{\epsilon}\left(\Gamma_h\right)
\right)}=0$:
\begin{eqnarray*}
  \int \!\!\! \int_{\boldsymbol{\Omega}}
\left(
\frac{\partial F_1^{\epsilon}}{\partial \eta_1}\right)^2
\hspace{2pt}
d\boldsymbol{\eta}
&=&
-\int \!\!\! \int_{\boldsymbol{\Omega}}
F_1^{\epsilon}
\frac{\partial^2 F_1^{\epsilon}}{\partial \eta_1^2}
\hspace{2pt}
d\boldsymbol{\eta}\\
&  &
+
\int_{\boldsymbol{\eta}\left(\phi^{\epsilon}\left(\Gamma_h\right)\right)}  
\left(F_1^{\epsilon}
\frac{\partial F_1^{\epsilon}}{\partial \eta_1},0\right)
\hspace{2pt}\cdot \hspace{2pt}
\mathbf{n}
\hspace{2pt} dS
=0.
\end{eqnarray*}
And, in the same way, Eq. (\ref{eq:2.43}) and (\ref{eq:2.45}) 
imply that:
\begin{eqnarray*}
\sum_{i=1,2}
  \int \!\!\! \int_{\boldsymbol{\Omega}}
\left(
\frac{\partial F_2^{\epsilon}}{\partial \eta_i}\right)^2
\hspace{2pt}
d\boldsymbol{\eta}
&=&
-\int \!\!\! \int_{\boldsymbol{\Omega}}
F_2^{\epsilon}
\Delta F_2
\hspace{2pt}
d\boldsymbol{\eta}
=0.
\end{eqnarray*}

Therefore, if Eq.  (\ref{eq:2.37}) is satisfied by $F^{\epsilon}$,
then $\mathcal{F}_1 \left(F^{\epsilon}\right)=\mathcal{F}_2 
\left(F^{\epsilon}\right)$,
and Eq. (\ref{eq:2.46}) gives:
\begin{eqnarray*}
\Vert \nabla F^{\epsilon} 
 \Vert_{L^2\left(\boldsymbol{\Omega};\mathbb{R}^2\right)}
 &=&
\int \!\!\! \int_{\boldsymbol{\Omega}}
\left(
\frac{\partial F_1^{\epsilon}}{\partial \eta_2}\right)^2
\hspace{2pt}
d\boldsymbol{\eta} \\
&\leq&
\frac{1}{2\hspace{2pt}C}
\hspace{2pt}
\int_{\boldsymbol{\eta}\left(\phi^{\epsilon}\left(\Gamma_h\right)\right)}   
\left([LU^{\epsilon}]^3,0\right)
\hspace{2pt}\cdot \hspace{2pt}
\mathbf{n}
\hspace{2pt} dS.
\end{eqnarray*}
Finally, each density value $\rho=\rho^{\epsilon}\leq c_2$,
and $\partial \eta_1/\partial s=\rho^{\epsilon}$ in
$\boldsymbol{\Omega}$.
Hence, 
\begin{eqnarray*}
\frac{1}{2\hspace{2pt}C}
\hspace{2pt}
\int_{\boldsymbol{\eta}\left(\phi^{\epsilon}\left(\Gamma_h\right)\right)}   
\left([LU^{\epsilon}]^3,0\right)
\hspace{2pt}\cdot \hspace{2pt}
\mathbf{n}
\hspace{2pt} dS
&\leq&
\frac{U^3}{2\hspace{2pt}C}
\int_{0}^{1}   
\frac{\partial \eta_1}{\partial s}\left(s,h^{\epsilon}\left(s\right)\right) 
\hspace{2pt} ds
\\
&\leq&
\frac{c_2\hspace{2pt}U^3}{2\hspace{2pt}C}.
\end{eqnarray*}
\end{proof}

\begin{theorem}\label{theo:2.4}
Without loss of generality, assume $L,H >1$.
Under the same hypothesis of Theorem~\ref{theo:2.2}, 
and the additional existence of locally integrable generalized 
derivatives up to order $2$ for $u$,
we obtain that $u$ is a weak solution to the limit formula:
\begin{eqnarray}\label{eq:2.47}
   f
    \hspace{2pt}
    \frac{\partial^2 u}{\partial y^2}
&=&
\frac{\partial f}{\partial y}
\hspace{2pt}
\frac{\partial u}{\partial y},
\end{eqnarray}
 in $L^2\left(\boldsymbol{\Omega_h};\mathbb{R}^2\right)$, 
 where 
$f=\left[1-\left( 
      u^2\left(x,y\right)/2i_0\right)\right]^{-\frac{6}{25}}
      $.
\end{theorem}

\begin{proof}
If $\mathbf{v}=(u,v) \in 
L^2\left(\boldsymbol{\Omega_h};\mathbb{R}^2\right)$,
and $L^2\epsilon=LH>1$:
\begin{eqnarray*}
\int \! \! \! \int_{\boldsymbol{\Omega}}
\left(F_1^{\epsilon}\left(\eta_1,\eta_2\right)\right)^2
\hspace{2pt} d\eta_1 d\eta_2
&=&
\int \! \! \! \int_{\boldsymbol{\Omega_{\epsilon}}}
\left(u^{\epsilon}\left(s,\tau \right)\right)^2
\hspace{2pt} ds \hspace{2pt} d\tau, \\
&=&
LH
\int \! \! \! \int_{\boldsymbol{\Omega_{h}}}
\left(u\left(x,y\right)\right)^2
\hspace{2pt} dx \hspace{2pt} dy, \\
&=& 
LH 
\hspace{2pt}
\Vert u\Vert_{L^2\left(\boldsymbol{\Omega_h}\right)}^2.
\end{eqnarray*}
In a similar manner, the estimate $\rho^{\epsilon}\leq c_2$ 
implies that:
\begin{eqnarray*}
\int \! \! \! \int_{\boldsymbol{\Omega}}
\left(F_2^{\epsilon}\left(\eta_1,\eta_2\right)\right)^2
\hspace{2pt} d\eta_1 d\eta_2
&=&
\int \! \! \! \int_{\boldsymbol{\Omega_{\epsilon}}}
\left(\left(\rho^{\epsilon}\right)^2
v^{\epsilon}\left(s,\tau \right)\right)^2
\hspace{2pt} ds \hspace{2pt} d\tau, \\
&\leq&
c_2^4
\hspace{2pt}
LH
\hspace{2pt}
\Vert v\Vert_{L^2\left(\boldsymbol{\Omega_h}\right)}^2.
\end{eqnarray*}

Therefore, 
\begin{eqnarray*}
\Vert F^{\epsilon}
\Vert_{W^{1,2}\left(\boldsymbol{\Omega};\mathbb{R}^2\right)}^2
&\leq&
L
\left\{
H 
\hspace{2pt}
\Vert u\Vert_{L^2\left(\boldsymbol{\Omega_h}\right)}^2
+
c_2^4
\hspace{2pt}
H
\hspace{2pt}
\Vert v\Vert_{L^2\left(\boldsymbol{\Omega_h}\right)}^2
+
\frac{c_2\hspace{2pt}U^3}{2\hspace{2pt}C}
\right\}.
\end{eqnarray*}
Thus, the sequence $\left(F^{\epsilon}\right)$
is contained and bounded in the Sobolev Space $W^{1,2}\left(\boldsymbol{\Omega};\mathbb{R}^2\right)$
by a constant value independent of the parameter $\epsilon>0$.
As a consequence, the Rellich-Kondrachov compactness theorem
A. Bressan \cite[p. 173, 178]{Bressan13} implies that it
has a subsequence $\left(F^{\epsilon_{\alpha}}\right)$  that
converges strongly in
$L^2\left(\boldsymbol{\Omega};\mathbb{R}^2\right)$,
and the sequence $\partial F_1^{\epsilon}/\partial \eta_2$
converges weakly in $L^2\left(\boldsymbol{\Omega}\right)$
to the gene\-ralized derivative
$\partial F_1/\partial \eta_2$
of the limit 
$F=(F_1,F_2) \in L^2\left(\boldsymbol{\Omega};\mathbb{R}^2\right)$.
But, $F_1^{\epsilon}=u^{\epsilon}=1/L \hspace{2pt} u$ for all 
$\epsilon>0$. Then, the horizontal velocity $u$ 
is a weak solution of the limit formula, Eq. (\ref{eq:2.47}), in 
$L^2\left(\boldsymbol{\Omega_h}\right)$
when the parameter $\epsilon$ tends to $0$.
\end{proof}

\section{Conclusion}\label{s3}

The obtained limit formula suggests that there is no separation of
the Boundary Layer under this conditions, but it shows that it is
possible to study the change of the horizontal velocity of
atmospheric wind with height near the surface by means of simpler
models. There are two immediate problems to work on:
First, to obtain solutions to the Reynolds' limit model
by the application of fractional calculus methods. Second,
to consider the case where the Neumann condition 
$\partial T/\partial z|_{z=0}=m$ is a constant $m \not =0$.

\section*{Acknowledgements}

 The corresponding author, Carla V. Valencia-Negrete, 
 thanks her institution, the 
 {\it Superior School of Physics and Mathematics} of the
 {\it National Polytechnic Institute}, (ESFM-IPN), 
 for the support, under Grant No $306221/275187$
 by \emph{National Council on Science and Technology} (CONACYT).


\end{document}